\documentclass[12pt]{article}

\usepackage{mathtools}
\usepackage{amssymb}
\usepackage{amsthm}
\usepackage{xcolor}
\usepackage{tikz}
\usepackage{fullpage}

\DeclareSymbolFont{bbold}{U}{bbold}{m}{n}
\DeclareSymbolFontAlphabet{\mathbbold}{bbold}

\newcommand{\N}{\mathbb{N}}

\newcommand{\R}{\mathbb{R}}

\newcommand{\1}{\mathbbold{1}_\Omega}

\newcommand{\Dir}{{\rm D}}
\newcommand{\Neu}{{\rm N}}

\makeatletter
\newcommand{\uD}{\@ifnextchar1{\u@DN{\Dir}{\mu_1}\@gobble}{\u@DN{\Dir}{\mu}}}
\newcommand{\uN}{\@ifnextchar1{\u@DN{\Neu}{\mu_1}\@gobble}{\u@DN{\Neu}{\mu}}}
\newcommand{\u@DN}[2]{u_{#1\mkern-1mu,\mkern1mu#2}}
\makeatother

\newcommand{\e}{{\rm e}}

\providecommand{\form}{\tau}
\providecommand{\scpr}[2]{\left( #1 \,\middle|\, #2 \right)}

\renewcommand{\sp}{\scpr}
\newcommand{\from}{\colon}

\let\phi\varphi

\let\leq\leqslant

\let\geq\geqslant

\makeatletter
\def\@row#1,{#1\@ifnextchar;{\@gobble}{&\@row}}
\def\@matrix{%
    \expandafter\@row\my@arg,;%
    \@ifnextchar({\\ \get@in@paren{\@matrix}}{\after@matrix}%
    }
\def\matrixtype#1#2#3{%
    \ifmmode\def\after@matrix{\end{#2}\right#3}%
    \else\def\after@matrix{\end{#2}\right#3$}$\fi\iffalse$\fi
    \left#1\begin{#2}\get@in@paren{\@matrix}%
    }
\def\@column#1,{#1\@ifnextchar;{\@gobble}{\\ \@column}}
\newcommand\vect{}
\def\svect(#1){\left(\begin{smallmatrix}\@column#1,;\end{smallmatrix}\right)}
\def\vect{\get@in@paren{\@vect}}
\def\@vect{\left(\begin{matrix}\expandafter\@column\my@arg,;\end{matrix}\right)}
\def\get@in@paren#1({\def\my@arg{}\def\my@rest{}\def\after@get{#1}\get@arg}
\let\e@a\expandafter
\def\get@arg#1){\e@a\kl@test\my@rest#1(;}
\def\kl@test#1(#2;{\e@a\def\e@a\my@arg\e@a{\my@arg#1}%
                   \ifx:#2:\let\my@exec\after@get
                   \else\let\my@exec\get@arg
                        \e@a\def\e@a\my@arg\e@a{\my@arg(}%
                        \def@rest#2;%
                   \fi\my@exec}
\def\def@rest#1(;{\def\my@rest{#1\kl@zu}}
\def\kl@zu{)}

\makeatletter
\newcommand\MyPairedDelimiter{%
  \@ifstar{\My@Paired@Delimiter{{}}}
          {\My@Paired@Delimiter{}}%
}
\newcommand\My@Paired@Delimiter[4]{%
  \newcommand#2{%
    \@ifstar{\start@PD{#1}{\delimitershortfall=-1sp}{#3}{#4}}
            {\start@PD{#1}{}{#3}{#4}}%
  }%
}
\newcommand\start@PD[5]{%
  #1\mathopen{\mathpalette\put@delim@helper{\put@delim{#2}{#3}{.}{#5}}}%
  #5%
  \mathclose{\mathpalette\put@delim@helper{\put@delim{#2}{.}{#4}{#5}}}%
}
\newcommand\put@delim@helper[2]{%
  \hbox{$\m@th\nulldelimiterspace=0pt #2#1$}%
}
\newcommand\put@delim[5]{%
  \setbox\z@\hbox{$\m@th#5{#4}$}%
  \setbox\tw@\null
  \ht\tw@\ht\z@ \dp\tw@\dp\z@
  #1#5%
  \left#2\box\tw@\right#3%
}

\makeatother
\MyPairedDelimiter*{\abs}{\lvert}{\rvert}
\MyPairedDelimiter*{\norm}{\lVert}{\rVert}
\MyPairedDelimiter{\set}{\{}{\}}

\theoremstyle{plain} % default
\newtheorem{theorem}{Theorem}[section]
\newtheorem{corollary}[theorem]{Corollary}

\theoremstyle{definition}
\newtheorem{example}[theorem]{Example}

\newtheorem{remark}[theorem]{Remark}

\usepackage{enumitem}

\setenumerate[1]{nolistsep} % Only the level 1
\setenumerate[2]{nolistsep} % Only the level 2

\begin{document}

\medmuskip=4mu plus 2mu minus 3mu
\thickmuskip=5mu plus 3mu minus 1mu
\belowdisplayshortskip=9pt plus 3pt minus 5pt

\title{On the perturbation of positive semigroups}

\author{Christian Seifert and Daniel Wingert}

\date{}

\maketitle

\begin{abstract}
 We prove a perturbation result for
 %sesquilinear forms generating 
 positive semigroups, 
 thereby extending a heat kernel estimate by Barlow, Grigor'yan and Kumagai for Dirichlet forms (\cite{bgk2009}) to positive semigroups. 
 This also leads to a generalization of domination for semigroups on $L_p$-spaces.
\end{abstract}



\section{Introduction}

In recent years there has been increasing interest in heat kernel estimates for non-local Dirichlet forms.
One important tool for deriving upper bounds is a perturbation result proven in \cite[Lemma 3.1]{bgk2009}, 
which is used in several recent articles \cite{bbck2009,ck2008,f2009,gh2008}. 
The aim of this paper is to prove this perturbation result in the more general context of positive semigroups, 
so that it can also be applied to perturbations of Dirichlet forms by (suitable) measures. 
They are of substantial interest for example in quantum mechanics.
In view of Banach lattices our result can be seen as a generalization of domination for positive semigroups; 
compare Theorem \ref{thm:perturbation} and \cite[Proposition C-II, 4.8]{Nagel1986}.

% \textbf{Nutzen wir das folgende?}
% Since Davies method \cite{d1987, cks1987}, the second main tool in these articles, works for these forms as well, 
% one obtains upper bounds on the heat kernel for Dirichlet forms perturbed by suitable measures with our generalized perturbation result.

In the remaining part of this section we explain the situation. 
The perturbation result is stated in Section \ref{sec:perturbation}, where also some examples are given.

Let $(\Omega,m)$ be a $\sigma$-finite measure space (in fact, it suffices to have a localizable space, where $L_1(m)' = L_\infty(m)$; cf.~\cite[Theorem 5.1]{Segal1951}).
Let $T\from[0,\infty)\to L(L_2(m))$ be a strongly continuous semigroup on $L_2(m)$ (equipped with the inner product $\sp{\cdot}{\cdot}$), i.e.,
\[T(0) = I,\quad T(s+t) = T(s)T(t) \quad(s,t\geq 0),\]
and
\[\lim_{t\to 0} T(t)u = u \quad(u\in L_2(m)).\]
Let $A$ be the generator of $T$, i.e.,
\[D(A):= \set{u\in L_2(m);\; \lim_{t\to 0}\frac{1}{t}(T(t)u-u) =:Au \;\text{exists}}.\]
Note that $A$ also uniquely defines $T$.

For $1\leq p<\infty$ the space $L_2(m)\cap L_p(m)$ is dense in $L_p(m)$. Thus, if there exist $M\geq 1$ and $\omega\in\R$ such that
\begin{equation}
\label{eq:exp_bound}
  \norm{T(t)u}_p \leq Me^{\omega t} \norm{u}_p \quad(u\in L_2(m)\cap L_p(m))
\end{equation}
then $T$ can be extended to a strongly continuous semigroup $T_p$ on $L_p(m)$.
We remark that $T_p(t)u = T(t) u$ for all $t\geq 0$ and $u\in L_2(m)\cap L_p(m)$.
Let $A_p$ be the generator of $T_p$.
Then, for $\lambda>\omega$, where $\omega$ is as in \eqref{eq:exp_bound} for $1\leq p,q<\infty$, we have $\lambda\in \varrho(A_p)\cap\varrho(A_q)$ and
\[(\lambda-A_p)^{-1} u = (\lambda-A_q)^{-1} u  \quad(u\in L_p(m)\cap L_q(m)).\]
The case $p=\infty$ is more delicate, since $L_2(m)\cap L_\infty(m)$ is in general not dense in $L_\infty(m)$. 
Therefore, consider the adjoint semigroup $T^*$ of $T$ generated by the adjoint $A^*$ of $A$, 
and let $T_\infty(t):= T_1^*(t)'$ (the dual operator of $(T^*)_1(t)$) for all $t\geq 0$. Note that $T_\infty$ is still a semigroup, 
but it may not be strongly continuous. However, for $\lambda$ sufficiently large we observe $\lambda\in\rho(A) \cap \rho(A_1^*)$ and
\[(\lambda-A)^{-1}u = ((\lambda-A_1^*)^{-1})'u \quad(u\in L_2(m)\cap L_\infty(m)).\]

\begin{remark}
\label{rem:extension}
  Let $1\leq p,q\leq \infty$, $B_p\in L(L_p(m))$, $B_q\in L(L_q(m))$ such that $B_p|_{L_p(m)\cap L_q(m)} = B_q|_{L_q(m)\cap L_q(m)}$. Then $B_p$ can be extended to $L_p(m)+L_q(m)$ via
  \[B_p(u+v):= B_p u + B_q v \quad(u\in L_p(m), v\in L_q(m)).\]
  This is indeed well-defined, since the operators are equal on the intersection.
\end{remark}
By the previous remark, if $T$ acts on $L_2(m)$ and $T^*$ acts on $L_1(m)$ then we can extend $T$ and the resolvent of $A$ to $L_2(m) + L_\infty(m)$.

Recall that $L_p(m)$ is also a Banach lattice for all $1\leq p\leq \infty$. A semigroup $T_p$ on $L_p(m)$ is called positive, 
if $T_p(t)u\geq 0$ for all $t\geq 0$ and $0\leq u\in L_p(m)$. For details concerning positive semigroups see e.g.~\cite{Nagel1986, o2005}.
\begin{remark}
  Let $1\leq p<\infty$, $T_p$ be a $C_0$-semigroup on $L_p(m)$ with generator $A_p$.
  Then $T_p$ is positive if and only if $(\lambda-A_p)^{-1}$ is positive for all $\lambda>\omega$, where $\omega$ is as in \eqref{eq:exp_bound}.
\end{remark}

\section{A perturbation result}
\label{sec:perturbation}

There is a perturbation result for the heat kernels of Dirichlet forms proven by Barlow, Grigor'yan and Kumagai 
in \cite[Lemma 3.1]{bgk2009} using Meyers decomposition of stochastic processes \cite{m1975}. 
Here we prove an analogous perturbation result within an analytic context. 
%To this end, the perturbation must satisfy a one sided ultracontractivity 
%and the according semigroups must be $L^1$ resp. $L^\infty$-bounded. 
More precisely, we show that such a perturbation result for semigroups is in fact equivalent to a certain condition on the
corresponding generators.
%one-sided ultracontracitvity condition.
This condition is fulfilled for the perturbations regarded in \cite{bgk2009}, 
however we do not assume the forms associated with the semigroups to be symmetric Dirichlet forms.
Our perturbation theorem also extends well-known results on domination of semigroups; cf.~Remark \ref{rem:domination} below.

\begin{theorem}\label{thm:perturbation}
  Let $T_0$ be a positive $C_0$-semigroup on $L_2(m)$ with generator $A_0$, $T$ a positive $C_0$-semigroup on $L_2(m)$ with generator $A$. 
  Assume %$D(A_0^*)\cap L_1$ is dense in $L_1\cap L_2$ and that 
  there exist $M\geq 1$ and $\omega\in\R$ such that
  \[\norm{T(t)u}_1 \leq Me^{\omega t} \norm{u}_1 \quad(u\in L_2(m)\cap L_1(m)),\; \norm{T_0(t)^*v}_1 \leq Me^{\omega t} \norm{v}_1 \quad(v\in L_2(m)\cap L_1(m))\]
  for all $t\geq 0$. 
  Then the following are equivalent:
  \begin{enumerate}
    \item
      There exists $C_1\in\R$ such that for $t\geq 0$ we have
      \[T(t) u \leq T_0(t) u + C_1 t e^{\omega t} \norm{u}_1 \1 \quad(0\leq u\in L_2(m)\cap L_1(m)).\]
    \item
      There exists $C_2\in\R$ such that
      \begin{equation}
      \label{eq:one-sided}
	\sp{Au}{v} \leq \sp{u}{A_0^* v} + C_2 \norm{u}_1\norm{v}_1 \quad(0\leq u\in D(A)\cap L_1(m), 0\leq v\in D(A_0^*)\cap L_1(m)).
      \end{equation}
    \item
      There exists $C_3\in\R$ such that for $\lambda>\omega$ we have
      \[(\lambda-A)^{-1} u \leq (\lambda-A_0)^{-1} u + \frac{C_3}{(\lambda-\omega)^2}\norm{u}_1 \1 \quad(0\leq u\in L_2(m)\cap L_1(m)).\]
  \end{enumerate}
\end{theorem}

\begin{remark}
  As the proof will show, in ``(a) $\Rightarrow$ (b)'' we have $C_2:=C_1$, in ``(b) $\Rightarrow$ (c)'' we have $C_3:=C_2M^2$ and in ``(c) $\Rightarrow$ (a)'' we have $C_1:=C_3M^2$.
\end{remark}

\begin{remark}\label{rem:domination}
  In case $C_1=C_2=C_3=0$ we recover well-known domination results, see e.g.\ \cite[Proposition C-II, 4.8]{Nagel1986} for domination of positive semigroups in Banach lattices.
\end{remark}

\begin{proof}[Proof of Theorem \ref{thm:perturbation}]
  ``(a) $\Rightarrow$ (b)'':
  Let $0\leq u\in D(A)\cap L_1(m)$, $0\leq v\in D(A_0^*)\cap L_1(m)$. Then, for $t>0$,
  \begin{align*}
    \sp{\frac{1}{t}(T(t)u-u)}{v} & = \sp{\frac{1}{t}T(t)u}{v} - \sp{\frac{1}{t}u}{v} \\
    & \leq \sp{\frac{1}{t}T_0(t)u}{v} + C_1 e^{\omega t} \norm{u}_1\norm{v}_1 - \sp{\frac{1}{t}u}{v} \\
    & = \sp{u}{\frac{1}{t}(T_0(t)^*v-v)} + C_1 e^{\omega t} \norm{u}_1\norm{v}_1.
  \end{align*}
  The limit $t\to 0$ yields the assertion.
  
  ``(b) $\Rightarrow$ (c)'':  
  Let $0\leq u\in L_2(m)\cap L_1(m)$, $0\leq v\in L_2(m)\cap L_1(m)$, $\lambda >\omega$. 
  Since $T$ and $T_0$ are positive, also the resolvents of $A$ and $A_0$ (and hence also of $A_0^*$) are positive.
  Thus $\tilde{u}:=(\lambda-A)^{-1}u, \tilde{v}:=(\lambda-A_0^*)^{-1}v\geq 0$. Furthermore, 
  $\tilde{u}\in D(A)\cap L_1(m)$ and $\tilde{v}\in D(A_0^*)\cap L_1(m)$.
  Hence,
  \begin{align*}
    \sp{(\lambda-A)^{-1}u - (\lambda-A_0)^{-1}u}{v} & = \sp{A\tilde{u}}{\tilde{v}} - \sp{\tilde{u}}{A_0^*\tilde{v}} \leq C_2\norm{\tilde{u}}_1\norm{\tilde{v}}_1.
  \end{align*}    
  Since $\norm{\tilde{u}}_1 = \norm{(\lambda-A)^{-1}u}_1 \leq \frac{M}{\lambda-\omega}\norm{u}_1$ and similarly for $\tilde{v}$, we obtain
  \[\sp{(\lambda-A)^{-1}u - (\lambda-A_0)^{-1}u}{v} \leq \frac{C_2M^2}{(\lambda-\omega)^2} \norm{u}_1\norm{v}_1.\]
  Thus,
  \[(\lambda-A)^{-1} u - (\lambda-A_0)^{-1} u \leq \frac{C_2M^2}{(\lambda-\omega)^2}\norm{u}_1 \1.\]
  
  ``(c) $\Rightarrow$ (a)'':  
  Let $0\leq u\in L_2(m)\cap L_1(m)$ and $\lambda>\omega$. 
  In view of Remark \ref{rem:extension}, by induction on $n$ we obtain
  \[(\lambda-A)^{-n} u \leq (\lambda-A_0)^{-n} u + \frac{C_3}{(\lambda-\omega)^2} \sum_{k=0}^{n-1} \norm{(\lambda-A)^{-k}u}_1 (\lambda-A_0)^{-n+1+k} \1.\]
%   Then
%   \[(\lambda-A)^{-1} u \leq (\lambda-A_0)^{-1} u + \frac{C_3}{(\lambda-\omega)^2}\norm{u}_1 \1 ,\]
%   and
%   \begin{align*}
%     (\lambda-A)^{-2} u & = (\lambda-A)^{-1}(\lambda-A)^{-1} u \leq (\lambda-A_0)^{-1} (\lambda-A)^{-1} u + \frac{C_3}{(\lambda-\omega)^2}\norm{(\lambda-A)^{-1} u}_1 \1 \\
%     & \leq (\lambda-A_0)^{-1} \bigl((\lambda-A_0)^{-1} u + \frac{C_3}{(\lambda-\omega)^2}\norm{u}_1 \1\bigr) + \frac{C_3}{(\lambda-\omega)^2} \frac{M}{\lambda-\omega} \norm{u}_1 \1 \\
%     & \leq (\lambda-A_0)^{-2} u + \frac{C_3}{(\lambda-\omega)^2}\norm{u}_1 (\lambda-A_0)^{-1} \1 + \frac{C_3}{(\lambda-\omega)^2} \frac{M}{\lambda-\omega} \norm{u}_1 \1\\
%     & \leq (\lambda-A_0)^{-2} u + \frac{C_3}{(\lambda-\omega)^2}\norm{u}_1 \frac{M}{\lambda-\omega} \1 + \frac{C_3}{(\lambda-\omega)^2} \frac{M}{\lambda-\omega} \norm{u}_1 \1 \\
%     & = (\lambda-A_0)^{-2} u + \frac{2C_3M}{(\lambda-\omega)^3}\norm{u}_1 \1 .
%   \end{align*}
%  Since $\norm{T_0(t)^*}_{1\to 1}\leq Me^{\omega t}$ we have $\norm{(\lambda-A_0)^{-1}}_{\infty\to \infty} \leq \frac{M}{\lambda-\omega}$.
  Note that $\norm{(\lambda-A)^{-k}u}_{1} \leq \frac{M}{(\lambda-\omega)^k}\norm{u}_1$ and $\norm{(\lambda-A_0^*)^{-k}u}_1 \leq \frac{M}{(\lambda-\omega)^k}\norm{u}_1$ for all $k\in\N$ by the Hille-Yoshida-Phillips Theorem.
  Hence, also $(\lambda-A_0)^{-k}\1 = ((\lambda-A_0^*)^{-k})' \1 \leq \frac{M}{(\lambda-\omega)^k} \1$ for $k\in\N$.
  Thus, we observe
  \[(\lambda-A)^{-n} u \leq (\lambda-A_0)^{-n} u + \frac{nC_3M^2}{(\lambda-\omega)^{n+1}}\norm{u}_1 \1 \quad(n\in\N).\]
  Let $t>0$. For large enough $n$ we have $\frac{n}{t} > \omega$. Hence,
  \begin{align*}
    (\tfrac{n}{t})^n(\tfrac{n}{t}-A)^{-n} u & \leq (\tfrac{n}{t})^n(\tfrac{n}{t}-A_0)^{-n} u + (\tfrac{n}{t})^n\frac{nC_3M^2}{(\tfrac{n}{t}-\omega)^{n+1}}\norm{u}_1 \1 \\
    & = (\tfrac{n}{t})^n(\tfrac{n}{t}-A_0)^{-n} u + \frac{(\tfrac{n}{t})^{n+1}}{(\tfrac{n}{t}-\omega)^{n+1}}tC_3M^2\norm{u}_1 \1 .
  \end{align*}
  For $n\to\infty$, we obtain
  \[T(t)u \leq T_0(t) u + e^{\omega t} tC_3M^2\norm{u}_1 \1. \qedhere\]
\end{proof}

We can now specialise to contraction semigroups.

\begin{corollary}
\label{cor:contr}
  Let $T_0$ be a positive $C_0$-semigroup on $L_2(m)$ with generator $A_0$, $T$ a positive $C_0$-semigroup on $L_2(m)$ with generator $A$.
  Assume that 
  \[\norm{T(t)u}_1 \leq \norm{u}_1 \quad(u\in L_2(m)\cap L_1(m)),\; \norm{T_0(t)^*v}_1 \leq \norm{v}_1 \quad(v\in L_2(m)\cap L_1(m))\]
  for all $t\geq 0$. Let $C\in\R$. Then the following are equivalent:
    \begin{enumerate}
    \item
      For $t\geq 0$ we have
      \[T(t) u \leq T_0(t) u + C t \norm{u}_1 \1 \quad(0\leq u\in L_2(m)\cap L_1(m)).\]
    \item
      We have
      \[\sp{Au}{v} \leq \sp{u}{A_0^* v} + C \norm{u}_1\norm{v}_1 \quad(0\leq u\in D(A)\cap L_1(m), 0\leq v\in D(A_0^*)\cap L_1(m)).\]
    \item
      For $\lambda>0$ we have
      \[(\lambda-A)^{-1} u \leq (\lambda-A_0)^{-1} u + \frac{C}{\lambda^2}\norm{u}_1 \1 \quad(0\leq u\in L_2(m)\cap L_1(m)).\]
  \end{enumerate}
\end{corollary}

% \begin{remark}
% 	Condition \eqref{eq:one-sided} can be seen as one-sided ultracontractivity. Indeed, if we have $A = A_0+B$ with some operator $B$, then 
% \textbf{Was meintest du damit? Das seh ich nicht.}
% % 	(b) \eqref{eq:one-sided} also includes a condition on $D(a)$ and $D(a_0)$. 
% % 	In general, $b$ is not equal to $a_0-a$. However, if $D(a) = D(a_0)$, then $b=a_0-a$.
% \end{remark}

In order to link our result to the version treated in \cite{bgk2009}, where estimates on the kernels were given, 
we state an easy consequence of Theorem \ref{thm:perturbation}.

% \begin{corollary}
% 	Under the assumptions of Theorem \ref{thm:perturbation}, let $(\e^{tA_0})_{t\geq0}$ have a kernel $k^0$, i.e., 
% 	there is $k^0\from [0,\infty)\times \Omega^2\to \R$ measurable such that
% 	\[\e^{-t A_0} u = \int_\Omega k^0_t(\cdot,y) u(y) \,m(\rmd y) \quad(u\in L^1\cap L^2(\Omega,m), t\geq 0). \]
% 	Then $(\e^{tA})_{t\geq0}$ has a kernel $k\from [0,\infty)\times \Omega^2\to \R$ and for $m^2$-a.a.~$(x,y)\in \Omega^2$ 
%  	\[k_t(x,y) \leq k^0_t(x,y) + C M^2 \e^{\omega t} t \quad (t\geq 0).\]
% \end{corollary}

\begin{corollary}
\label{cor:kernel}
  Let $T_0$ be a positive $C_0$-semigroup on $L_2(m)$ with generator $A_0$, $T$ a positive $C_0$-semigroup on $L_2(m)$ with generator $A$. 
  Assume there exist $M\geq 1$ and $\omega\in\R$ such that
  \[\norm{T(t)u}_1 \leq Me^{\omega t} \norm{u}_1 \quad(u\in L_2(m)\cap L_1(m)),\; \norm{T_0(t)^*v}_1 \leq Me^{\omega t} \norm{v}_1 \quad(v\in L_2(m)\cap L_1(m))\]
  for all $t\geq 0$, and there exists $C\in\R$ such that
      \[
	\sp{Au}{v} \leq \sp{u}{A_0^* v} + C \norm{u}_1\norm{v}_1 \quad(0\leq u\in D(A)\cap L_1(m), 0\leq v\in D(A_0^*)\cap L_1(m)).
      \]
  Let $T_0$ have a kernel $k^0$, i.e.,
  $k^0\from [0,\infty)\times \Omega^2\to [0,\infty)$ is measurable such that
  \[T_0(t) u = \int_\Omega k^0_t(\cdot,y) u(y) \,dm(y) \quad(u\in L_2(m), t\geq 0). \]
  Then $T$ has a kernel $k$ satisfying 
  \[k_t(x,y) \leq k^0_t(x,y) + C M^4 \e^{\omega t} t \quad (m^2\text{-a.a.}\; (x,y)\in \Omega^2, t\geq 0).\]
\end{corollary}

Note that the existence of the kernel follows from \cite[Theorem 5.9]{aa2002}, 
the kernel estimate then from the estimate for the semigroups in Theorem \ref{thm:perturbation} and the monotone convergence theorem; see also \cite[Korollar 2.1.11]{w2011}.

\begin{example}
We will now apply our theorem to Dirichlet forms perturbed by measures and jump processes.
Let $(\Omega,m)$ be a locally compact metric measure space and $\form_0$ a regular Dirichlet form on $L_2(m)$ 
Let $\mu$ be a non-negative Borel measure on $\Omega$ which is absolutely continuous with respect to $\form_0$-capacity.
Furthermore, let $\mu$ be in the extended Kato class with Kato bound less than $1$. 
Define $\form_{-\mu}:=\form_0-\mu$, let $A_{-\mu}$ be the operator associated with $\form_{-\mu}$ and $T_{-\mu}$ the $C_0$-semigroup (for perturbation of Dirichlet forms by measures we refer to \cite{sv1996}).
Then $A_{-\mu}$ is self-adjoint and by \cite[Theorem 3.3]{sv1996} the semigroup $T_{-\mu}$ satifies \eqref{eq:exp_bound} for all $1\leq p<\infty$.
Let $j\from \Omega\times\Omega\to \R$ be measurable, symmetric, $0\leq j\leq C$.
Define $\form$ by
\[\form(u,v):= \form_{-\mu}(u,v) + \frac{1}{2} \int_{\Omega^2} (u(x)-u(y)) (v(x)-v(y)) j(x,y) \, dm^2(x,y),\]
and $A$ be the associated operator and $T$ the $C_0$-semigroup. Then, for $0\leq u\in D(A)\cap L_1(m)$ and $0\leq v\in D(A_{-\mu}^*)\cap L_1(m)$ we compute
\begin{align*}
  \sp{-Au}{v} - \sp{u}{-A_{-\mu}^*v} & = -\frac{1}{2} \int_{\Omega^2} (u(x)-u(y)) (v(x)-v(y)) j(x,y) \, dm^2(x,y) \\
  & \leq \int u(x)v(y)j(x,y)\,dm^2(x,y) \leq C\norm{u}_1\norm{v}_1.
\end{align*}
Thus, by Corollary \ref{cor:kernel} we obtain an estimate for the kernel of the semigroup $T$ by the kernel for $T_{-\mu}$ (of course, only if there exists a kernel for $T_{-\mu}$). Note that the case $\mu=0$ recovers the result in \cite[Lemma 3.1]{bgk2009}.
\end{example}

\begin{example}
% 	(a) Let $b$ be of the form 
% 	\[b(u,v) = - \int_{\Omega^2} u(y) \overline{v(x)} \,\mu(\rmd x, \rmd y),\]
% 	where $\mu$ is a positive measure on $\Omega^2$. Then \eqref{eq:one-sided} is satisfied with $C=0$.
% 
% 	(b) Let $b$ be of the form
%  	\[ b(u, v) = \int_{\Omega^2} u(y) \overline{v(x)} k(x, y) \,m^2(\rmd x, \rmd y),\]
%  	where $k\from \Omega^2\to [0,C]$ is measurable. Then \eqref{eq:one-sided} is satisfied.

% 	(c) 
We will also provide a counterexample.
Define
	\begin{align*}
		D(\form_0) & := H^1(0,1),\\
		\form_0(u,v) & := \int_0^1 u'(x) v'(x)\, dx, \\
		D(\form) & := \set{u\in H^1(0,1);\; u(0) = u(1)},\\
		\form(u,v) & := \form_0(u,v).
	\end{align*}
	Let $A$ and $A_0$ be the associated operators with $\form$ and $\form_0$, respectively 
	($A$ is the Laplace with Neumann boundary conditions and $A_0$ the Laplace with periodic boundary conditions).
	Then, for $u\in D(A)$ and $v\in D(A_0^*) = D(A_0)$, an easy calculation shows
	\[\sp{-Au}{v} - \sp{u}{-A_0^*v} = u'(0)\bigl(v(1)-v(0)\bigr).\]
	Hence, \eqref{eq:one-sided} can hardly be satisfied. Therefore, perturbations by boundary conditions may behave differently.
\end{example}
  
% \begin{example}
% 	\begin{enumerate}
% 		\item
% 			 If $a$ and $a_0$ are Dirichlet forms they satisfy $\norm{\e^{-t A}}_1 \leq 1$ and $\norm{\e^{-t A_0}}_\infty \leq 1$. If $a$ and $a_0$ are measure perturbed Dirichlet forms, with the negative measure in the appropriate Kato class, they satisfy $\norm{\e^{-t A}}_1 \leq M \e^{\omega t}$ and $\norm{\e^{-t A_0}}_\infty \leq M \e^{\omega t}$ for some $M$ and $\omega$ (see \cite{sv1996}).
% 		\item
% 			 If $b$ is of the type
% 			\[ b(f, g) = \lim_{n \rightarrow \infty} \int_{\Omega_n \times \Omega_n} (f(x)-f(y)) g(x) j(x, y) \:m^2(\rmd x, \rmd y) \]
% 			with $0 \leq j \leq C$ measurable and $\Omega_n \uparrow \Omega$, $b$ satisfies the one sided ultracontractivity condition with the constant $C$. The symmetric form
% 			\[ b(f, g) = \frac{1}{2} \int_{\Omega \times \Omega} (f(x)-f(y)) (g(x)-g(y)) j(x, y) \:m^2(\rmd x, \rmd y) \]
% 			with $j$ symmetric is a special case of this. This symmetric type of perturbation was regarded in \cite{bgk2009} together with Dirichlet forms.
% 		\item
% 			Let $\mu$ be an absolutely continuous measure on $\Omega$ ($\mu \ll m$) and
% 			\[ b(f, g) = \int_\Omega f(x) g(x) \:\mu(\rmd x). \]
% 			Then $b$ satisfies the one sided ultracontractivity condition with $C = 0$. However, for such a perturbation $b$ the result is well-known. More generally spoken, every positive form clearly satisfies the one sided ultracontractivity condition with constant $C = 0$.
% 	\end{enumerate}
% \end{example}

\section*{Acknowledgements}

The authors thank Peter Stollmann, Daniel Lenz and Hendrik Vogt for many stimulating conversations.

\noindent
Christian Seifert \\
Technische Universit\"at Hamburg-Harburg \\
Institut f\"ur Mathematik \\
21073 Hamburg, Germany \\
{\tt christian.se\rlap{\textcolor{white}{hugo@egon}}ifert@tu-\rlap{\textcolor{white}{darmstadt}}harburg.de}

\medskip

\noindent
Daniel Wingert \\
Technische Universit\"at Chemnitz \\
Fakult\"at f\"ur Mathematik \\
09107 Chemnitz, Germany \\
{\tt daniel.win\rlap{\textcolor{white}{hugo@egon}}gert@mathematik.tu-\rlap{\textcolor{white}{darmstadt}}chemnitz.de}


\begin{thebibliography}{99}

\bibitem{aa2002}
Y.A.\;Abramovich and C.D.\;Aliprantis,
\emph{An invitiation to operator theory.}
Graduate Studies in Mathematics, 50, American Mathematical Society, Procivence, RI, 2002.



\bibitem{Nagel1986}
W. Arendt, A. Grabosch, G. Greiner, U. Groh, H. Lotz, U. Moustakas, R. Nagel, F. Neubrander and U. Schlotterbeck,
\emph{One-parameter Semigroups of Positive Operators.}
Springer, 1986. 

\bibitem{bbck2009}
M.T.\;Barlow, R.F.\;Bass, Z.Q.\;Chen and M.\;Kassmann,
\emph{Non-local Dirichlet forms and symmetric jump processes.}
Trans.\ Amer.\ Math.\ Soc.\ \textbf{361}(4) (2009), 1963--1999.

\bibitem{bgk2009}
M.T.\;Barlow, A.\;Grigor'yan and T.\;Kumagai,
\emph{Heat kernel upper bounds for jump processes and the first exit time.}
J.\ Reine Angew.\ Math.\ \textbf{626} (2009), 135--157.

% \bibitem{cks1987}
% E.A.\;Carlen, S.\;Kusuoka and D.W.\;Stroock,
% \emph{Upper bounds for symmetric Markov transition functions.}
% Ann.\ Inst.\ H.\ Poincar\'{e}\ \textbf{23} (1987), 245--287.

\bibitem{ck2008}
Z.Q.\;Chen and T.\;Kumagai,
\emph{Heat kernel estimates for jump processes of mixed types on metric measure spaces.}
Probab.\ Theory Relates Fields\ \textbf{140}(1-2) (2008), 277--317.

% \bibitem{d1987}
% E.B.\;Davies,
% \emph{Explicit constants for Gaussian upper bounds on heat kernels.}
% Amer.\ J.\ Math.\ \textbf{109} (1987), 319-334.

\bibitem{f2009}
M.\;Foondun,
\emph{Heat kernel estimates and Harnack inequalities for some Dirichlet forms with non-local part.}
Electron.\ J.\ Probab.\ \textbf{14}(11) (2009), 314--340.

\bibitem{gh2008}
A.\;Grigor'yan and J.\;Hu,
\emph{Off-diagonal upper estimates for the heat kernel of the Dirichlet forms on metric spaces.}
Invent.\ Math.\ \textbf{174}(1) (2008), 81--126.

% \bibitem{k2008}
% K.\;Kuwae,
% \emph{Maximum Principles for Subharmonic Functions Via Local Semi-Dirichlet Forms.}
% Canad.\ J.\ Math.\ \textbf{60}(4) (2008), 822--874.

% \bibitem{mr1995}
% Z.-M.\;Ma and M.\;R\"ockner,
% \emph{Markov Processes associated with positivity preserving coercive forms.}
% Can.\ J.\ Math.\ \textbf{47}(4) (1995), 817--840.

\bibitem{m1975}
P.-A.\;Meyer,
\emph{Renaissance, recollectement, m\'elanges, ralentissement de processus de Markov.}
Ann.\ Inst.\ Fourier\ \textbf{25} (3-4) (1975), 464--497.

\bibitem{o2005}
E.M.\;Ouhabaz,
\emph{Analysis of Heat Equations on Domains.}
Princeton\ Univ.\ Press, Princeton, NJ, 2005.

\bibitem{Segal1951}
I.E.\;Segal,
\emph{Equivalences of Measure Spaces}.
Am.\ J.\ Math.\ \textbf{73}(2) (1951), 275--313.


\bibitem{sv1996}
P.\;Stollmann and J.\;Voigt,
\emph{Perturbation of Dirichlet Forms by Measures.}
Pot.\ An.\ \textbf{5} (1996), 109--138.

% \bibitem{s1994}
% K.-Th.\;Sturm,
% \emph{Analysis on local {D}irichlet spaces. {I}. {R}ecurrence, conservativeness and {$L^p$}-{L}iouville properties.}
% J.\ Reine Angew.\ Math.\ \textbf{456} (1994), 173--196.

\bibitem{w2011}
D.\;Wingert,
\emph{Evolutionsgleichungen und obere Absch\"atzungen an die L\"osungen des Anfangswertproblems.}
Doctoral Thesis (2011). url:\\
\texttt{http://nbn-resolving.de/urn:nbn:de:bsz:ch1-qucosa-107849}.

\end{thebibliography}
\end{document}